\documentclass{article}
\usepackage[utf8]{inputenc}
\usepackage[english]{babel}
\usepackage{authblk,bbm}
\usepackage{amssymb}
\usepackage{amsthm}
\usepackage[colorlinks,urlcolor=blue,linkcolor=black]{hyperref}
\usepackage{fancyhdr}
\usepackage{tikz}
\usepackage{verbatim}
\usepackage[osf]{mathpazo}
\usepackage{amsmath,amsfonts,graphicx}
\usepackage{latexsym}
\usepackage[top=1in,bottom=1.4in,left=1.5in,right=1.5in,centering]{geometry}
\usepackage{color}
\usepackage{clrscode3e}
\usepackage[colorinlistoftodos]{todonotes}
\usepackage{float}
\usepackage{setspace,caption,subcaption,fullpage}
\theoremstyle{definition}
\newtheorem{define}{Definition}[section]
\newtheorem{theorem}{Theorem}[section]
\newtheorem*{theorem*}{Theorem}

\newtheorem{cor}{Corollary}[section]
\theoremstyle{remark}

\makeatletter
\renewcommand*\env@matrix[1][\arraystretch]{%
  \edef\arraystretch{#1}%
  \hskip -\arraycolsep
  \let\@ifnextchar\new@ifnextchar
  \array{*\c@MaxMatrixCols c}}

\title{Technical Report: Improved Fourier Reconstruction using Jump Information with Applications to MRI}

\author[1]{Jade Larriva-Latt\thanks{jlarriva@wellesley.edu}}
\author[2]{Angela Morrison\thanks{arm14@albion.edu}}
\author[3]{Alison Radgowski\thanks{alrad001@mail.goucher.edu}}
\author[4]{Joseph Tobin\thanks{jat9kf.@virginia.edu}}
\author[5]{Aditya Viswanathan\thanks{aditya@math.msu.edu}}
\author[6]{Mark Iwen\thanks{markiwen@math.msu.edu}}

\affil[1]{Department of Mathematics, Wellesley College}
\affil[2]{Department of Mathematics and Computer Science, Albion College}
\affil[3]{Department of Mathematics and Computer Science, Goucher College}
\affil[4]{Department of Mathematics, University of Virginia}
\affil[5]{Department of Mathematics, Michigan State University}
\affil[6]{Dept. of Mathematics, and Dept. of ECE, Michigan State University}

\begin{document}

\maketitle

\begin{abstract}
Certain applications such as Magnetic Resonance Imaging (MRI) require the reconstruction of
functions from Fourier spectral data. When the underlying functions are piecewise-smooth, standard
Fourier approximation methods suffer from the {\em Gibbs phenomenon} --  with associated oscillatory
artifacts in the vicinity of edges and an overall reduced order of convergence in the approximation.
This paper proposes an edge-augmented Fourier reconstruction procedure which uses only the first few
Fourier coefficients of an underlying piecewise-smooth function to accurately estimate jump
information and then incorporate it into a Fourier partial sum approximation. We provide
both theoretical and empirical results showing the improved accuracy of the proposed method,
as well as comparisons demonstrating superior performance over existing state-of-the-art sparse
optimization-based methods.  Extensions of the proposed techniques to functions of several variables are also addressed preliminarily.
All code used to generate the results in this report can be found at \cite{Bitbucket}.
\end{abstract}

\section{Introduction}
This paper addresses the problem of reconstructing a $2\pi$-periodic piecewise-smooth function given
a small number of its lowest frequency Fourier series coefficients.  The Fourier partial sum
reconstruction of such piecewise-smooth functions suffers from the {\em Gibbs phenomenon}
\cite{Hesthaven_2007_Spectral} -- with its associated non-physical oscillations in the vicinity of
jump discontinuities and an overall reduced order of accuracy in the reconstruction. In applications
such as MR imaging -- where the scanning apparatus collects Fourier coefficients
\cite{Nishimura_PrinciplesMRI_2010} of the specimen being imaged -- these oscillatory artifacts and
the reduced order of accuracy are significant impediments to generating accurate and fast scan
images. Hence, there exists significant ongoing and inter-disciplinary interest in novel methods of
reconstructing such functions from Fourier spectral data.

\subsection{Related Work}
\label{sec:lit_survey}
The traditional approach to mitigating Gibbs artifacts is using low-pass filtering
\cite{Hesthaven_2007_Spectral}. However, this does not completely eliminate all artifacts; filtered
reconstructions still suffer from smearing in the vicinity of edges and improved convergence rates
are restricted to regions away from edges. Spectral reprojection methods such as
\cite{Archibald_GegenbauerMRI_2002} work by reconstructing the function in each smooth interval
using an alternate (non-periodic) basis such as those consisting of Gegenbauer polynomials. While
these methods have been shown to be highly accurate, they are sensitive to parameter choice; indeed,
small errors in parameter selection or estimated edge location (which are used to determine the
intervals of smoothness) can lead to large reconstruction errors. More recently, there has been
significant interest in compressed sensing \cite{Candes_Stable_2006,Donoho_Compressed_2006} based
approaches to this problem -- see, for example, \cite{Lustig_Sparse_2007,Lustig_Compressed_2008}.
While these approaches are indeed extremely powerful and versatile, they are implicitly discrete
methods and do not perform well when provided with continuous measurements as is the case here.
Indeed, we show in the empirical results below that they exhibit poor (first-order) numerical
convergence when recovering a piecewise-smooth function $f$ from its {\em continuous} Fourier
measurements. The method proposed in this paper is perhaps most closely
related to those in 
\cite{Batenkov_2015,ongie2015recovery,ongie2015super,vetterli2002sampling,maravic2005sampling,haacke1989superresolution,haacke1989constrained} 
all of which use prony-like methods to estimate the 
 jumps in $f$ (after which the jumps' deleterious spectral effects can be mitigated).
In contrast to these previous approaches, however, herein we $(i)$ propose the use of a simple spectral extrapolation scheme together with an alternate (and highly noise robust) 
non-prony-based jump estimation procedure, $(ii)$ provide ($\ell_2$-norm) error
analysis of the proposed method for general piecewise-smooth functions, $(iii)$ present comparisons with state-of-the-art sparse
optimization based reconstruction methods, {\it and} $(iv)$ present preliminary 2D reconstruction results.

The rest of this paper is organized as follows: In \S \ref{sec:notation}, we set up notation and
provide some definitions before relating jump information to the Fourier coefficients of a
piecewise-smooth function. \S \ref{sec:JumpAugPartSumRec} shows how this jump information can be
incorporated into a modified Fourier partial sum approximation. Some analytic error bounds for this
jump augmented Fourier reconstruction are also provided. Next, \S \ref{Sec:NumMethJumpAugSum}
defines two different methods of accurately estimating jump information given Fourier spectral data.
\S \ref{sec:error_bounds_est} provides theoretical error bounds for the proposed reconstruction
method with estimated edge information while \S \ref{sec:2D} discusses an extension of the method to
two-dimensional problems. We provide some concluding comments and directions for future research in
\S \ref{sec:conclusion}.

\section{Notation, Setup, and Review}
\label{sec:notation}

We will utilize the following characterization of piecewise smooth functions of one variable.

\begin{define}
A $2\pi$-periodic function $f: \mathbbm{R} \rightarrow \mathbbm{R}$ is {\it piecewise continuous} if
\begin{enumerate}
\item $f$ is Riemann integrable on $[-\pi, \pi]$, and
\item $f$ is continuous at all points in $(-\pi,\pi]$ except for at most finitely many points $$-\pi < x_{1} < \dots < x_{J} \leq \pi.$$
\label{def:piecewiseContin}
\end{enumerate}
\end{define}

\begin{define}
A $2\pi$-periodic function $f: \mathbbm{R} \rightarrow \mathbbm{R}$ is {\it piecewise smooth} if
\begin{enumerate}
\item $f$ is piecewise continuous, 
\item $f$ is differentiable on all of $(-\pi, x_1)$, $(x_1, x_2)$, $\dots$, $(x_{J-1}, x_J)$, $(x_J, \pi)$,
\item $f'$ as a $2 \pi$-periodic function with $f'(x_1) := \dots := f'(x_J) := 0$ is also piecewise continuous (having potentially more than $J$ discontinuities in $(-\pi, \pi]$),
\item $f'$ is differentiable almost everywhere on all of the open intervals between its neighboring discontinuities, and
\item $f''$ is Lebesgue integrable on $[-\pi, \pi]$.
\end{enumerate}
\label{def:piecewiseSmooth}
\end{define}

The Fourier series of a $2\pi$-periodic piecewise smooth function $f$ is
\begin{equation*}
 \sum_{k = -\infty}^{\infty} \hat f_{k} e^{ikx}
\end{equation*}
where the $\hat f_{k}$, the $k^{\rm th}$ Fourier coefficient of $f$, is
\begin{equation}
\hat f_{k} = \frac{1}{2 \pi} \int_{-\pi}^{\pi} f(x)e^{-ikx} dx.
\label{Fourier Coefficients}
\end{equation}

Note that computing the full Fourier series is not always feasible in practice.  When this is impossible, a partial Fourier sum  
\begin{equation}
S_{N}f(x) := \sum_{k = -N}^{N} \hat f_{k} e^{ikx}
\label{Standard Reconstruction}
\end{equation}
can be used for a given $N \in \mathbbm{N}$. It is well known that, as $N$ increases, $S_{N}f$ converges to $f$ pointwise almost everywhere. For example, we have the following theorem.

\begin{theorem}(See \cite{Folland_1992})
If $f: \mathbbm{R} \rightarrow \mathbbm{R}$ is piecewise smooth then $$\lim_{N\to\infty} S_Nf(x)=\frac{1}{2}[f(x^-)+f(x^+)]$$ for every $x \in \mathbbm{R}$, where $f(x^-) := \displaystyle{\lim_{\delta \rightarrow 0 {\rm ~for~} \delta >0} f(x - \delta)}$ and $f(x^+) := \displaystyle{\lim_{\delta \rightarrow 0 {\rm ~for~} \delta >0} f(x + \delta)}$.  In particular, $\displaystyle \lim_{N\to\infty}S_Nf(x)=f(x)$ for every $x$ at which $f$ is continuous.  
\label{thm:PointwiseApproxSN}
\end{theorem}

Given that we can expect pointwise convergence of partial Fourier sums as $N \rightarrow \infty$, our next concern becomes determining how large $N$ needs to be before we obtain an accurate approximation of $f$.  This, in turn, will inevitably lead one to study the decay of $\hat f_k$ as $k \rightarrow \pm \infty$.

\subsection{The Fourier Coefficients of Piecewise Smooth Functions with Jumps}


Let $f$ be a piecewise smooth function with discontinuities (or {\it jumps}) at $-\pi <  x_1 < \dots < x_J \leq \pi$.  These $x_j$-values will also be called {\it jump locations}.  Considering the $k^{\rm th}$ Fourier coefficient of $f$, 
$$\hat f_{k} = \frac{1}{2 \pi} \int_{-\pi}^{\pi} f(x)e^{-ikx} dx,$$
and splitting up the integral at the jump locations we can see that 
$$ \hat f_{k} = \frac{1}{2 \pi} \int_{-\pi}^{x_1} f(x) e^{-ikx} dx + \left( \sum^{J-1}_{j=1} \int_{x_j}^{x_{j+1}}f(x) e^{-ikx} dx \right) + \int_{x_J}^{\pi}f(x) e^{-ikx} dx.$$
Using integration by parts $J+1$ times we can now see that 
\begin{align*}
\hat f_{k} = \frac{-1}{2\pi i k} \Bigg\{ &f(x) e^{-ikx} \biggr \vert_{-\pi^+}^{x_{1}^{-}} - \int_{-\pi}^{x_1} f^{\prime}(x) e^{-ikx} dx + \left( \sum^{J-1}_{j=1} f(x) e^{-ikx} \biggr \vert_{x_j^+}^{x_{j+1}^{-}} - \int_{x_j}^{x_{j+1}} f^{\prime}(x) e^{-ikx} dx \right) \\ 
&+f(x) e^{-ikx} \biggr \vert_{x_{J}^{+}}^{\pi^-} - \int_{x_J}^{\pi} f^{\prime}(x) e^{-ikx} dx\Bigg\}.
\end{align*}
Collecting terms, and recalling that $f$ is $2 \pi$-periodic, we learn that
\begin{align}
\hat f_{k} = \frac{-1}{2\pi i k}\Bigg\{ &\left[ \sum^{J}_{j=1} \left( f(x_{j}^{-})-f(x_{j}^{+}) \right) e^{-ik x_{j}} \right] \nonumber \\ 
&- \left( \sum^{J-1}_{j=1} \int_{x_j}^{x_{j+1}}f'(x) e^{-ikx} dx \right) - \int_{-\pi}^{x_1}f'(x)e^{-ikx} dx-\int_{x_J}^{\pi}f'(x)e^{-ikx} dx\Bigg\}. \label{equ:FDecayPSfuncs}
\end{align}
This brings us to the following useful definition.

\begin{define}
The {\it jump function} of a piecewise continuous function $f$ is defined by
$$[f](x) := f(x^{+})-f(x^{-})$$
for all $x \in \mathbbm{R}$.  The value $[f](x)$ will also be referred to as the {\it jump height at $x$}. 
\label{Def:JumpFunc}
\end{define}

Through a second use of integration by parts and Definition~\ref{Def:JumpFunc} on the integrals \eqref{equ:FDecayPSfuncs} we arrive at the following well known result.

\begin{theorem}
If $f: \mathbbm{R} \rightarrow \mathbbm{R}$ is piecewise smooth then $$\left| \hat f_{k} - \sum^{J}_{j=1} \frac{[f](x_{j})}{2\pi i k} e^{-ik x_{j}} \right| \leq \frac{C}{k^2}$$
holds for all $k \in \mathbbm{Z} \setminus \{0 \}$, where $C \in \mathbbm{R}^+$ is an absolute constant that only depends on $f''$ and the jump function of $f'$.  As a result, both $\hat f_{k} = \mathcal{O}(1/|k|)$ and $\hat f_{k} \sim \sum^{J}_{j=1} \frac{[f](x_{j})}{2\pi i k} e^{-ik x_{j}}$ are true.
\label{thm:CoefEst}
\end{theorem}

The relatively slow $\mathcal{O}(1/k)$-decay of the Fourier coefficients of piecewise smooth functions with jumps is closely related to Gibbs phenomenon.   In general, the slower $\hat f_k$ decays as $|k| \rightarrow \infty$, the slower $S_{N}f(x)$ from \eqref{Standard Reconstruction} will converge to $f$ as $N$ increases (around points of discontinuity in particular).  

\subsection{Gibbs Phenomenon, and Two Examples}
\label{sec:gibbs&Examples}

Recall that the Fourier coefficients, $\hat f_{k}$, of any piecewise smooth function $f$ with jump discontinuities at $x_1, \dots, x_J$ will exhibit relatively slow $\mathcal{O}(1/k)$-decay (e.g., see Theorem~\ref{thm:CoefEst}).  Despite this fact, Theorem~\ref{thm:PointwiseApproxSN} guarantees that $\lim_{N\to\infty} S_Nf(x) = f(x)$ for all $x \notin \{ x_1, \dots, x_J \}$.  However, this convergence of $S_Nf(x)$ to $f(x)$ for almost all $x \in [-\pi, \pi]$ is not quite as nice as it looks.  It turns out that there exists an absolute constant $c \in (0,1)$ such that for every $N \in \mathbbm{N}$ one can find many $x$ near each jump location $x_j$ with $\left |S_Nf(x) - f(x) \right| > c [f](x_j)$.  In short, no matter how large you make $N$, there are guaranteed to be values of $x$ near $f$'s jump locations where $S_Nf(x)$ approximates $f$ badly.    This unfortunate fact is known as \textit{Gibbs phenomenon}.


As an example, let
\begin{equation}
    h(x) = \left\{
    \arraycolsep=4pt\def\arraystretch{1.5}
    \begin{array}{cc}
        \frac{3}{2} & \frac{-3\pi}{4} \leq x <\frac{-\pi}{2}  \\
        \frac{7}{4}-\frac{x}{2} +\text{sin}(x -\frac{1}{4}) & \frac{-\pi}{4} \leq x <\frac{\pi}{8} \\
         \frac{11}{4}x-5 & \frac{3\pi}{8} \leq x <\frac{3\pi}{4} \\
         0 & \text{else.} \\
    \end{array}
    \right.
    \label{equ:piecewise}
\end{equation}
Note that $h$ has six jump discontinuities at $x_1 = \frac{-3\pi}{4}$, $x_2 = \frac{-\pi}{2}$, $x_3 = \frac{-\pi}{4}$, $x_4 = \frac{\pi}{8}$, $x_5 = \frac{3\pi}{8}$, and $x_6 = \frac{3\pi}{4}$.  See Figure~\ref{fig:h(x)} for its graph.

\begin{figure}[H]
    \centering
    \includegraphics[scale=.8]{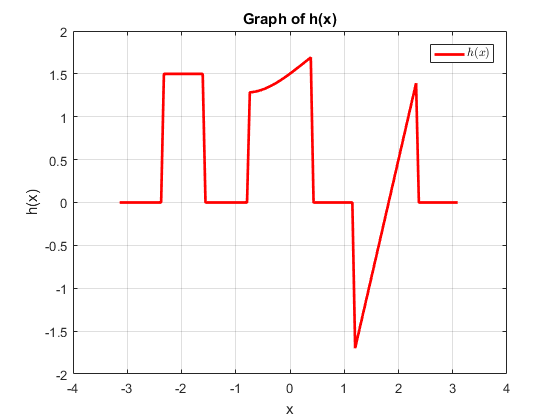}
    \caption{Graph of $h(x)$}
    \label{fig:h(x)}
\end{figure}

Another example of a piecewise smooth function with three jump discontinuities is 
$$s(x)=\begin{cases}
    x^2 & -\pi \leq x \leq \frac{-\pi}{2} \\ e^{x+3} & \frac{-\pi}{2} < x  \leq \frac{\pi}{2} \\
    e^4 x & \frac{\pi}{2} < x \leq \pi
\end{cases}.$$
Note that $x_3 = \pi$ counts as the third jump location for $s$, with the first two at $x_1 = \frac{-\pi}{2}$ and $x_2 = \frac{\pi}{2}$.  See Figure~\ref{fig:s(x)} for its graph.

\begin{figure}[H]
    \centering
    \includegraphics[scale=.6]{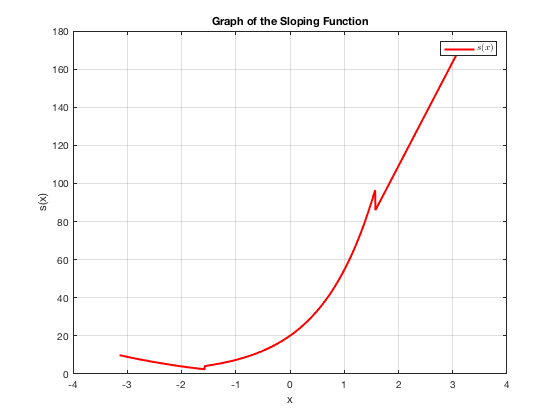}
    \label{piecewise}
    \caption{Graph of $s(x)$}
    \label{fig:s(x)}
\end{figure}

Gibbs phenomenon implies that functions with multiple jumps, such as $h$ in \eqref{equ:piecewise},
generally require many Fourier coefficients in order to be accurately approximated using standard
methods {\it even away from their jumps}.  As seen in Figures \ref{Standard Compare} and \ref{diff n
error}, oscillations around the discontinuities are not reduced as $N$ tends to infinity.
Furthermore, although the oscillations do decrease {\it away} from each jump discontinuity as $N$
grows, they decrease rather slowly.  In the case of MR imaging, jump discontinuities in the function
being measured might, e.g., correspond to tissue boundaries in a patient's body (e.g., bone to
muscle, or from cartilage to fluid). Thus, Gibbs tends to most severely degrade one's ability to
accurately image the boundaries between different tissues. In order to image such regions more
accurately, the effects of Gibbs phenomenon must be reduced. 

\begin{figure}[H]
    \centering
    \includegraphics[scale = .8]{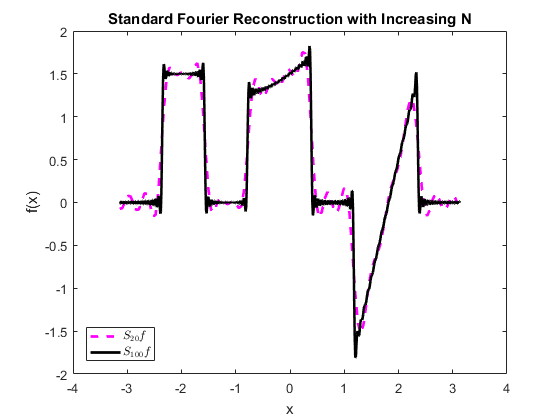}
    \caption{Partial Sum Approximation of $h$ from \eqref{equ:piecewise} for $N = 20$ and $N= 100$.}
    \label{Standard Compare}
\end{figure}

\begin{figure}[H]
    \centering
    \includegraphics[scale = .6]{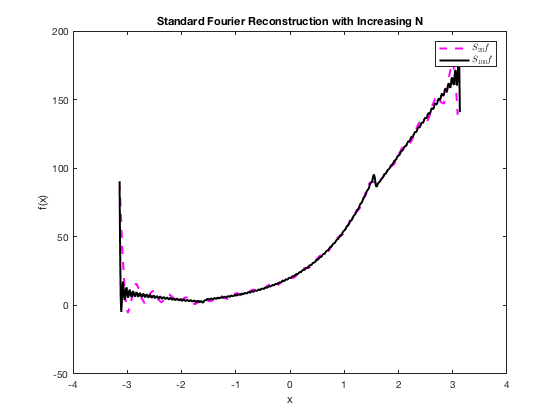}
    \caption{Partial Sum Approximation of $s$ for $N = 20$ and $N= 100$.}
    \label{Standard Compare2}
\end{figure}

\begin{figure}[H]
    \centering
    \includegraphics[scale=.8]{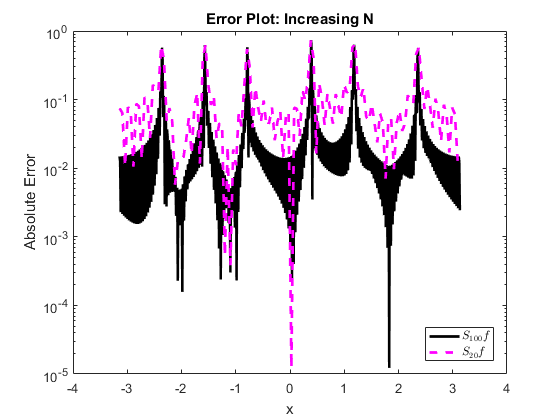}
    \caption{Absolute error comparison between standard reconstruction using $N=20$ and $N=100$ for $h(x)$.}
    \label{diff n error}
\end{figure}

\begin{figure}[H]
    \centering
    \includegraphics[scale=.6]{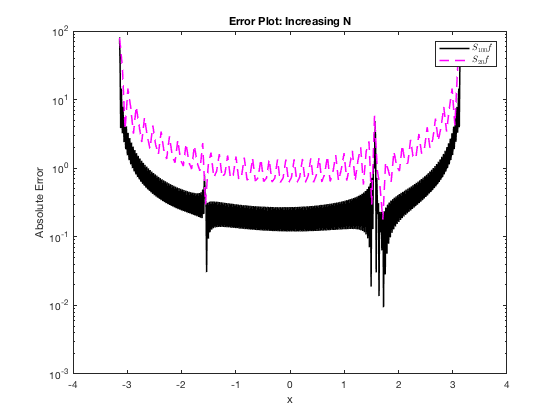}
    \caption{Absolute error comparison between standard reconstruction using $N=20$ and $N=100$ for $s(x)$.}
    \label{diff n error2}
\end{figure}

\section{A Jump Augmented Partial Sum Reconstruction}
\label{sec:JumpAugPartSumRec}



Assuming that the jump locations $x_1, \dots, x_J$ and jump heights $[f](x_{1}), \dots, [f](x_{J})$ of $f$ are known, one can use Theorem~\ref{thm:CoefEst} to estimate the Fourier coefficients of $f$ for all $k \in \mathbbm{Z} \setminus \{0 \}$ by 
\begin{equation}
\hat f_{k}^{est} := \sum_{j = 1}^{J} \frac{[f](x_j)}{2\pi ik}e^{-ik x_{j}}.
\label{Jump Coeffs}
\end{equation}
We can now augment \eqref{Standard Reconstruction} by incorporating these jump-based estimates. The resulting augmented partial sum approximation, $S_{N}^{edge}f$, is defined by 
\begin{equation}
S_{N}^{edge}f(x) := S_{N}f(x) + \sum_{|k| > N}\hat f_{k}^{est} e^{ikx} = \sum_{|k|\leq N}\hat f_{k} e^{ikx} +\sum_{|k| > N}\hat f_{k}^{est} e^{ikx}
\label{Edge 1}
\end{equation}
for all $x \in \mathbbm{R}$.  Note that $S_{N}^{edge}f$ still only utilizes $2N+1$ true Fourier coefficients of $f$.


We would like to use as many terms from the last sum in \eqref{Edge 1} as we can.  Toward this end, let's consider the form of the complete last sum with $\hat f_{0}^{est} := 0$.  We have that
\begin{align*}
\sum^{\infty}_{k = -\infty} \hat f_{k}^{est} e^{ikx} &= \sum_{0<|k|<\infty} \left( \sum_{j = 1}^{J}\frac{[f](x_{j})}{2\pi ik}e^{-ikx_{j}} \right) e^{ikx} \\
&=\sum_{j=1}^{J}[f](x_{j}) \left( \sum_{0<|k|<\infty} \frac{e^{-ikx_j}}{2\pi ik}e^{ikx} \right).
\end{align*}
Note that the $k^{\rm th}$ Fourier coefficient of the $2\pi$-periodic {\it ramp function} $r_j(x)$, defined by
\begin{equation}\label{ramp}
r_j(x) :=\begin{cases}
    \frac{-\pi-x}{2\pi}, & x<x_j \\
    \frac{\pi-x}{2\pi}, & x>x_j
\end{cases}
\end{equation}
for all $x \in [-\pi, \pi]$, is given by $(\widehat{r_j})_k = \frac{e^{-ikx_j}}{2\pi ik}$ for all $k \in \mathbbm{Z} \setminus \{0\}$.\footnote{See Figure~\ref{ramp function} for an example ramp function.}  Also, $(\widehat{r_j})_0 = 0$.
\begin{figure}
\centering
\includegraphics[scale = .8]{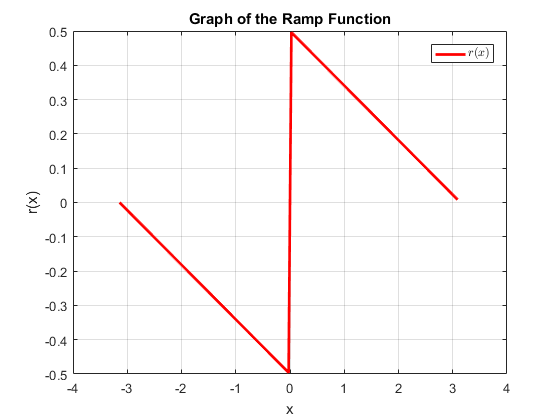}
\caption{Ramp function $r_j(x)$ with jump at $x_j=0$.}
\label{ramp function}
\end{figure}
Thus, we have that 
$$\sum^{\infty}_{k = -\infty} \hat f_{k}^{est} e^{ikx} = \sum_{j=1}^{J}[f](x_{j})r_{j}(x).$$
We are now able to give a more easily computable closed for expression for $S_{N}^{edge}f$ by noting that
\begin{equation}
S_{N}^{edge}f(x)=\sum_{k = -N}^{N} (\hat{f}_{k}-\hat{f}^{est}_{k})e^{ikx} +\sum_{j=1}^{J}[f](x_j)r_j(x).
\label{Edge 2}
\end{equation}

Figure \ref{True Edge} shows this edge-augmented reconstruction using true jump information, $S_{N}^{edge}f$, compared to the standard reconstruction, $S_{N}f$. Notice the great reduction in the Gibbs Phenomenon in the edge-augmented reconstruction. Another feature to take note of is that even thought the standard reconstruction uses more coefficients, the edge-augmented reconstruction is still more accurate (see Figure~\ref{standard v edge1} for an absolute error plot). 

\begin{figure}[H]
    \centering
    \includegraphics[scale=.8]{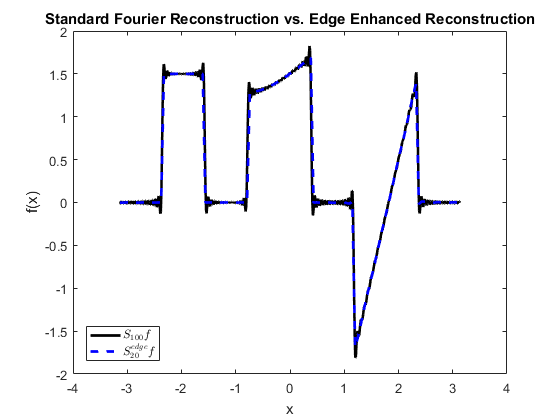}
    \caption{Standard Fourier Reconstruction, $S_{N}f$, compared to the edge-enhanced reconstruction, $S_{N}^{edge}f$.  The standard method uses $N = 100$ coefficients while the new reconstruction uses only $20$, in addition to true jump information.}
    \label{True Edge}
\end{figure}

\begin{figure}[H]
    \centering
    \includegraphics[scale=.6]{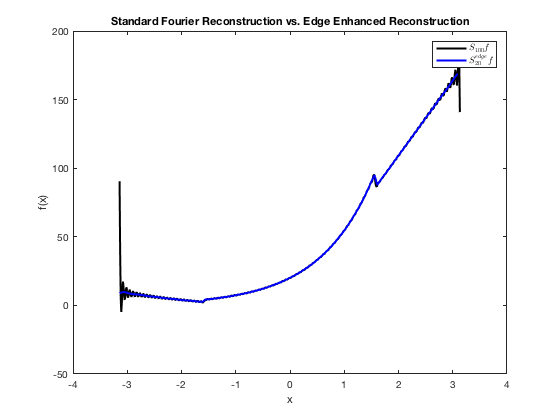}
    \caption{Standard Fourier Reconstruction, $S_{N}f$, of sloping function compared to the edge-enhanced reconstruction, $S_{N}^{edge}f$. The standard method uses $N=100$ coefficients while the new reconstruction uses only $20$, in addition to true jump information.}
    \label{Sloping True Edge}
\end{figure}

\begin{figure}[H]
    \centering
    \includegraphics[scale = .8]{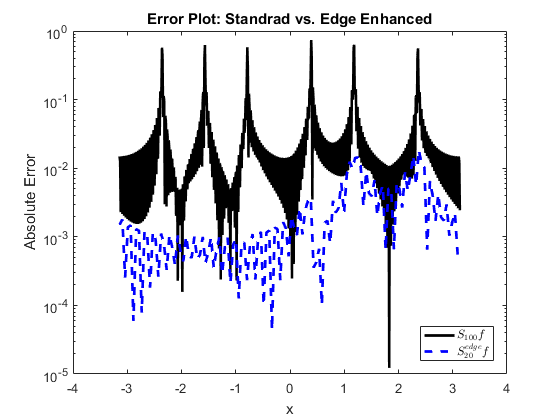}
    \caption{Error plot for Figure \ref{True Edge}}
    \label{standard v edge1}
\end{figure}

\begin{figure}[H]
    \centering
    \includegraphics[scale = .6]{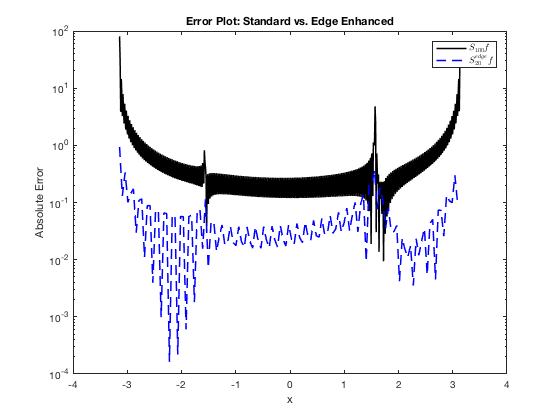}
    \caption{Error plot for Figure \ref{Sloping True Edge}}
    \label{standard v edge2}
\end{figure}

As Figures~\ref{standard v edge1}~and~\ref{standard v edge2} demonstrate, the jump-augmented sum $S_{N}^{edge}f(x)$ is numerically superior to the standard partial sum $S_{N}f$ when it comes to approximating both of our example functions from Section~\ref{sec:gibbs&Examples}.  In the next subsection we will prove that this is indeed the case more generally.

\subsection{Analytical Error Bounds:  Jump Information Helps!}
\label{sec:BasicThErrorBounds}

We will begin by providing error bounds for $\| f - S_{N}f \|_2$ in Section~\ref{sec:analStandardError}.  These error bounds are well know -- however, we include them here for the sake of completeness.  Then, in Section~\ref{sec:analJumpAugError}, we will provide error bounds for $\| f - S_{N}^{edge}f \|_2$.  As we will see, $S_{N}^{edge}f$ is guaranteed to converge to $f$ in norm much more quickly than $S_{N}f$ does as $N \rightarrow \infty$.  In fact, convergence is fast enough that $S_{N}^{edge}f$ is guaranteed to converge to $f$ uniformly as $N \rightarrow \infty$.

\subsubsection{Error Bound for $\| f - S_{N}f \|_2$}
\label{sec:analStandardError}

We will begin with a simple theorem that relates the two-norm error $\| f - S_{N}f \|_2$ to the decay of the Fourier coefficients of $f$.

\begin{theorem}
Let $f: \mathbbm{R} \rightarrow \mathbbm{C}$ be a $2\pi$-periodic $L^2([-\pi,\pi])$ function with Fourier coefficients satisfying $|\hat{f}_k|\leq\frac{c}{|k|^p}$ for a given $p \in (\frac{1}{2}, \infty)$ and $c \in \mathbbm{R}^+$ that are both independent of $k$.  Let $S_Nf$ be the partial Fourier sum approximation of $f$ as per \eqref{Standard Reconstruction}. 
Then, 
$$||f -S_Nf ||_2\leq \sqrt{\frac{2c^2}{(2p-1)N^{2p-1}}}.$$
\label{thm:BasicPSapprox}
\end{theorem}


\begin{proof}
We have that 
$$\begin{aligned}
\|f -S_Nf \|^2_2 & =  \sum_{|k|>N}| \hat{f}_k|^2 \hspace{.7in} &\textrm{(by Parseval's Theorem)}\\
& \leq \sum_{|k|>N}\left( \frac{c}{|k|^p} \right)^2 & \\ 
& \leq 2c^2 \int_{N}^{\infty}\frac{1}{x^{2p}}~dx & \\
&= \frac{2c^2}{(2p-1)N^{2p-1}}.
\end{aligned}$$
Taking square roots now establishes the desired result.
\end{proof}

Theorem~\ref{thm:BasicPSapprox} immediately provides us with the following corollary.
 
\begin{cor}
If $f$ is piecewise smooth then $$\|f-S_Nf\|_2\leq\sqrt{\frac{2 c^2}{N}}$$
for some constant $c \in \mathbbm{R}^+$ which is independent of $N$. 
\label{cor:PS2normError}
\end{cor} 

\begin{proof}
Theorem~\ref{thm:CoefEst} tells us that we may apply Theorem~\ref{thm:BasicPSapprox} with $p=1$.
\end{proof}


The astute reader will be able to see that Corollary~\ref{cor:PS2normError} is actually quite sharp in the case of piecewise smooth functions with jump discontinuities (i.e., there is a similar asymptotic {\it lower bound} for $||f-S_Nf||_2$ with respect to $N$ for such functions).  Given this, we are now ready to show that $S_{N}^{edge}f$ converges to $f$ in norm faster than $S_Nf$ does for all piecewise smooth functions with jump discontinuities.

\subsubsection{Error Bounds for $\| f - S_{N}^{edge}f \|_2 $}
\label{sec:analJumpAugError}

Recall from \eqref{Jump Coeffs} and \eqref{Edge 1} that 
$$S_N^{edge}f(x)=\sum_{|k|\leq N}\hat{f}_ke^{ikx}+\sum_{N<|k|<\infty} \hat{f}_k^{est}e^{ikx}$$ 
where
$$\hat f_{k}^{est} = \sum_{j = 1}^{J} \frac{[f](x_j)}{2\pi ik}e^{-ik x_{j}}.$$
We can establish an error bound for $\| f - S_{N}^{edge}f \|_2 $ using a very similar proof to that of Theorem~\ref{thm:BasicPSapprox}.



\begin{theorem}\label{fsnedgeerror}
If $f$ is piecewise smooth then  
$$||f-S_N^{edge}f||_2\leq\sqrt{\frac{2c^2}{3N^3}}$$
for some constant $c \in \mathbbm{R}^+$ which is independent of $N$. 
\end{theorem} 
  
\begin{proof}
We have that
$$\begin{aligned}
\|f -S_N^{edge}f \|^2_2 & =  \sum_{|k|>N}| \hat{f}_k -\hat{f}_k^{est}|^2 \hspace{.7in} &\textrm{(by Parseval's Theorem)}\\
& \leq \sum_{|k|>N}\left( \frac{C}{k^2} \right)^2 & \textrm{(By Theorem~\ref{thm:CoefEst}.)}\\
& \leq 2C^2 \int_{N}^{\infty}\frac{1}{x^{4}}~dx & \\
&= \frac{2C^2}{3 N^{3}}. & 
\end{aligned}$$
Taking square roots now yields the desired result.
\end{proof}


Similar to Corollary~\ref{cor:PS2normError} above, a very astute reader will be able to see that
Theorem~\ref{fsnedgeerror} is sharp with respect to its asymptotic $N$-dependence in the case of
piecewise smooth functions that have, e.g., jump discontinuities in their first derivatives.  As an
example, Figure \ref{true and standard two norm error} plots the $L^2$-error of each reconstruction
method for the function $h$ from \eqref{equ:piecewise} as $N$ increases. Notice that the slope for
the standard reconstruction error curve, $\log \left(\| h - S_Nh \|_2 \right)$, in Figure \ref{true
and standard two norm error} is about $-\frac{1}{2}$ which is consistent with the exponent on $N$ in
Corollary~\ref{cor:PS2normError}.  Similarly, the slope of the jump-augmented reconstruction error
curve, $\log \left(\|h - S_N^{edge}h \|_2 \right)$, in Figure \ref{true and standard two norm error}
is about $-\frac{3}{2}$ as expected from Theorem~\ref{fsnedgeerror}.

%
\begin{figure}[h]
\centering
\includegraphics[scale = .8]{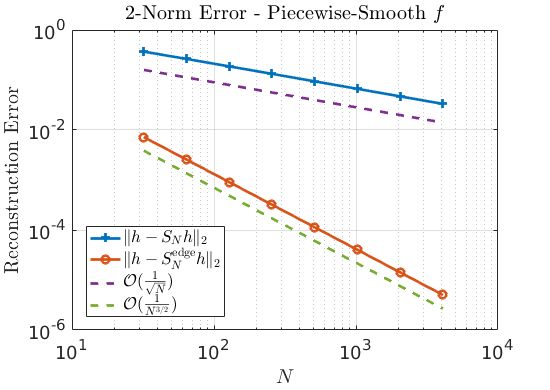}
\caption{Two norm error of the standard reconstruction, $S_Nh$, compared to the two norm error of the jump-augmented reconstruction using true jump information, $S_N^{edge}h$, for the function $h$ from \eqref{equ:piecewise}.}
\label{true and standard two norm error}
\end{figure}

 The challenge now is that the jump-augmented Fourier sum approximation of $f$, $S_N^{edge}f$, requires one to know what the true jump heights and jump locations of $f$ are.  However, this information is often unavailable. When an MRI scan is done, for example, one only has access to a few Fourier coefficients. What is required, then, is a method for finding estimates of a given piecewise smooth function's jump locations and jump heights, as well as the total number of jumps, that only uses a few Fourier coefficients of the function $f$ as input. Two methods that can help us do this are $(i)$ Prony's Method, and $(ii)$ concentration kernels with conjugate sums. Each of these are briefly described in the following section. 
  
\section{Numerical Methods for Approximately Computing $S_{N}^{edge}f$ Using Just a Few Fourier Coefficients}
\label{Sec:NumMethJumpAugSum}
 
 In this section we briefly summarize two different approaches that one can use in order to estimate the jump locations and heights present in a piecewise smooth function using only its first few Fourier series coefficients.  We will begin by describing a simple Prony-type method for the problem.
  
\subsection{Prony's Method}
\label{sec:prony}
Prony's method is a technique for transforming a non-linear parameter fitting problem into a
sequence of linear and root-finding problems that are easier to solve. Below we will briefly
demonstrate how Prony's Method can be used to estimate the jump heights and jump locations of a
piecewise smooth function using only a few of its Fourier coefficients.  For a simple introduction
to Prony's method see, e.g., \cite[Chapter 2]{Prony}.

Fix a parameter $T$.  When given sufficiently many observed values, $y[k]$, Prony's method solves a system of equations of the form  
\begin{align} 
y[k] &=  \sum_{j=1}^{J} C_j  \left( e^{(\sigma_{j}+i2\pi f_{j})T} \right)^{k},\quad \textrm{with} \quad k = 1,2,3,...
\label{equ:PronySysSolved}
\end{align}
for the unknowns $C_j, \sigma_{j},$ and $f_{j}$, for all $j = 1, \dots, J$.
In order to use Prony's method for our jump height and jump location estimation problem we will make these equations look like the ones approximately provided by Theorem~\ref{thm:CoefEst}.

Using Theorem~\ref{thm:CoefEst} and neglecting $\mathcal{O}(1/k^2)$-terms, we have that
\begin{equation}
\hat f_{k} \approx \sum_{j=1}^{J}\frac{[f](x_{j})}{2\pi ik}e^{-ikx_{j}}, \quad \textrm{with} \quad k = 1,2,...,N.
\end{equation}
Multiplying through by $2\pi i k$ we can see that
\begin{equation*}
\left( 2\pi ik \right) \hat f_{k} \approx \sum_{j=1}^{J} [f](x_{j})e^{-ikx_{j}}
\end{equation*}
holds for all $k \neq 0$.
Now our unknowns can be renamed to match those of the system that Prony's method solves \eqref{equ:PronySysSolved} as follows:
Use the parameter $T = 1$, and let $C_j= [f](x_{j}) \in \mathbbm{R}$, $\sigma_j = 0$, and $2\pi f_j = x_j$  for all $j = 1, \dots, J$.  Finally, take our observed values to be
\begin{equation} 
y[k] = (2\pi ik) \hat f_{k}.
\label{y}
\end{equation}
Given this setup a Prony method can now be used in order to estimate all $J$ jump locations $x_1,
\dots, x_J$ and $J$ jump heights $[f](x_{1}), \dots, [f](x_{J})$ using only the measured Fourier
coefficients $\hat f_{k}$.  See \cite[Chapter 2]{Prony} for details.

Figure \ref{Prony's} compares the jump-augmented reconstruction method of
Section~\ref{sec:JumpAugPartSumRec} using estimated jump locations and jump heights from Prony's
Method against the same method using true jump information. It performs relatively well despite the
small number of Fourier coefficients utilized.  See Figure \ref{jump_error} for an absolute error
plot.  Prony's method still has some drawbacks, however. For example, it needs to have a relatively
accurate estimate for the total number of true jumps, $J$, in order for the method to work well. For
this reason, among others, we also looked into the use of concentration kernels and conjugate sums
for jump location and height estimation.

\begin{figure}[H]
    \centering
    \includegraphics[scale=.8]{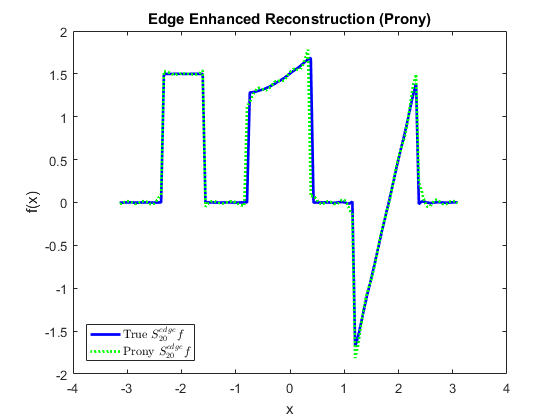}
    \caption{Comparing the use of true jump information to that of the estimates obtained by Prony's Method}
    \label{Prony's}
\end{figure} 

\begin{figure}[H]
    \centering
    \includegraphics[scale=.6]{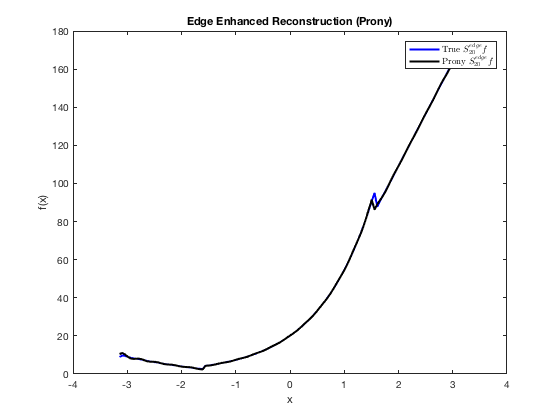}
    \caption{Comparing the use of true jump information to that of the estimates obtained by Prony's Method}
    \label{Sloping Prony's}
\end{figure} 

\subsection{Concentration Kernels}
\label{sec:conc}
Identifying jump information from spectral data is a challenging task since it involves the
extraction of local features from global data. In particular, Fourier reconstructions of
piecewise-smooth functions suffer from Gibbs oscillations which can often be mistaken for jumps by
conventional edge detectors, such as those based on divided differences. The concentration method of
edge detection \cite{conc_method_Gelb_1999, conc_method_Gelb_2000}, which was specifically devised
to address this problem, computes an approximation to the jump function $[f]$ (Def.
\ref{Def:JumpFunc}) using a filtered Fourier partial sum of the form 
\begin{equation}
  \label{eq:conc_sum}
  K_N^\sigma[f](x) := \sum_{|k| \leq N} \widehat f_k \left( i \, 
  \mbox{sgn}(k) \, \sigma\left( |k|/N \right) \right) e^{ikx}.
\end{equation}
Here, $\sigma$, known as a {\em concentration factor}, defines a special class of edge detection
``filters''. Under certain admissibility conditions (see \cite{conc_method_Gelb_2000} for details),
the Fourier partial sum (\ref{eq:conc_sum}) ``concentrates'' at the singular support of $f$ and we
have the following property (\cite[Theorem 2.3]{conc_method_theorem_2008}):
\[ K_N^\sigma[f](x) = [f](x) + \left \lbrace 
\begin{array}{ll}
    \mathcal O \left(  \frac {\log N}N \right),  & 
            d(x) \lesssim \frac{\log N}N \\
    \mathcal O \left(  \frac {\log N}{(N d(x))^s}\right),  & 
            d(x) \gg \frac 1 N, \\
\end{array}
\right.\]
where $d(x)$ denotes the distance between $x$ and the nearest jump discontinuity and $s = s_\sigma >
0$ depends on the choice of $\sigma$.
%
%
\begin{table}[!t]
\renewcommand{\arraystretch}{1.3}
\caption{Examples of concentration factors}
\label{tab:cfac}
\centering
\begin{tabular}{|c|c|}
\hline
\textbf{Factor} & \textbf{Expression} \\
\hline
Trigonometric & $\displaystyle \sigma_{T}(\eta) = 
	\frac{\pi \sin (\alpha \, \eta)}{Si(\alpha)}$ \\
 & $\displaystyle Si(\alpha) = \int _{0}^{\alpha} \frac{\sin (x)}{x} \, dx$ \\
\hline
Polynomial & $\displaystyle \sigma_{P}(\eta) = p \, \pi \, \eta^{p}$ \\
& $p$ is the order of the factor\\
\hline
Exponential & $ \sigma_{E}(\eta) = 
	C \eta \exp \left[ \frac{1}{\alpha \, \eta \, (\eta - 1)} \right]$\\
& $C$ - normalizing constant\\
& $\alpha$ - order \\
& $C = \frac{\pi}{\int _{\frac{1}{N}} ^{1-\frac{1}{N}}
	\exp \left[ \frac{1}{\alpha \, \tau \, (\tau - 1)} \right]d \tau }$\\
\hline
\end{tabular}
\end{table}
\begin{figure}[!t]
\centering
\includegraphics[trim = 42.5mm 80mm 30mm 100mm, clip, 
        scale=.5]{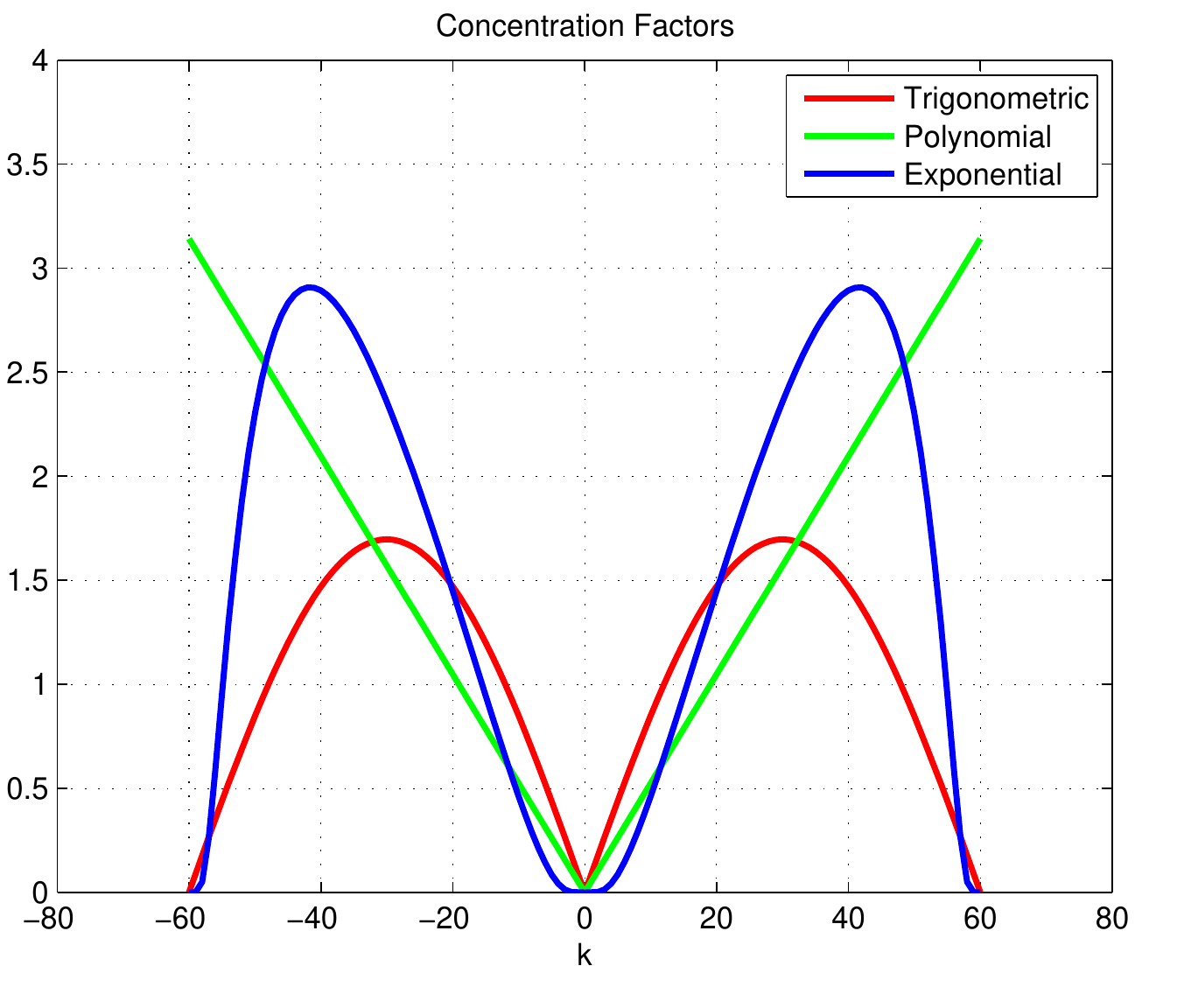} 
\caption{Envelopes of Concentration Factors in Fourier space}
\label{fig:cfac}
\end{figure}
Some common families of concentration factors and their corresponding plots in Fourier space are
provided in Table \ref{tab:cfac} and Fig. \ref{fig:cfac} respectively.  The jump function
approximation produced by each concentration factor $\sigma$ is unique and has certain
characteristic features. For example, the first order Polynomial concentration factor $\sigma_P,
P=1$ produces a jump approximation which is equivalent to computing a scaled first derivative of the
Fourier partial sum $S_Nf$. A jump function approximation using the Trigonometric concentration factor
utilizes the approximately $9\%$ Gibbs overshoots and undershoots on either side of a jump to
determine both jump locations and heights. Finally, the exponential concentration factor $\sigma_E$
is designed so as to yield root-exponentially small jump approximations away from the jumps. 

Once the jump approximation $K^\sigma_N[f]$ -- which can be evaluated efficiently using FFTs -- has
been computed, the jump locations and values can be identified using simple thresholding and peak
finding. This provides estimates of jump locations accurate to the nearest grid point used in the
evaluation of the concentration sum (\ref{eq:conc_sum}). However, we note that for well separated
jumps (at least $\mathcal O(\log N/N)$ away from each other), $K_N^\sigma [f]$ is a locally convex
function and the jump location estimates can be improved by implementing a local gradient-descent
based optimization procedure. Alternatively, a non-linear optimization routine (such as one based on
the Levenberg-Marquardt algorithm) can be employed in conjunction with Theorem \ref{thm:CoefEst} to
obtain sub-grid accuracy for jump locations and jump values. In Fig. \ref{fig:conc_method_example},
a representative jump function approximation using $N=50$ Fourier coefficients of the
piecewise-smooth test function (\ref{equ:piecewise}) and the trigonometric concentration factor is
plotted. Matlab's {\tt findpeaks} was used to perform peakfinding and the initial jump location
estimates were refined using {\tt fsolve} to implement a non-linear fitting of Theorem
\ref{thm:CoefEst}. While we defer analysis of the sub-grid accuracy of the concentration method to
future research, we present empirical evidence of this in Figure
\ref{fig:jump_conv}, where the error in estimating jump locations and values is plotted as a
function of the number of Fourier modes $N$. For reference, results using the Prony procedure from
\S \ref{sec:prony} -- which assumes that the number of jumps is known apriori -- is also plotted.
The figure shows that both jump locations and values can be estimated at an accuracy $\mathcal
O(1/N^2)$ even though we are given only $\mathcal O(N)$ Fourier coefficients. 

\begin{figure}[htbp]
    \centering
    \begin{subfigure}[b]{0.475\textwidth}
        \includegraphics[trim = 1.2in 3.2in 1.5in 2.5in, 
            clip=true, scale = .55]{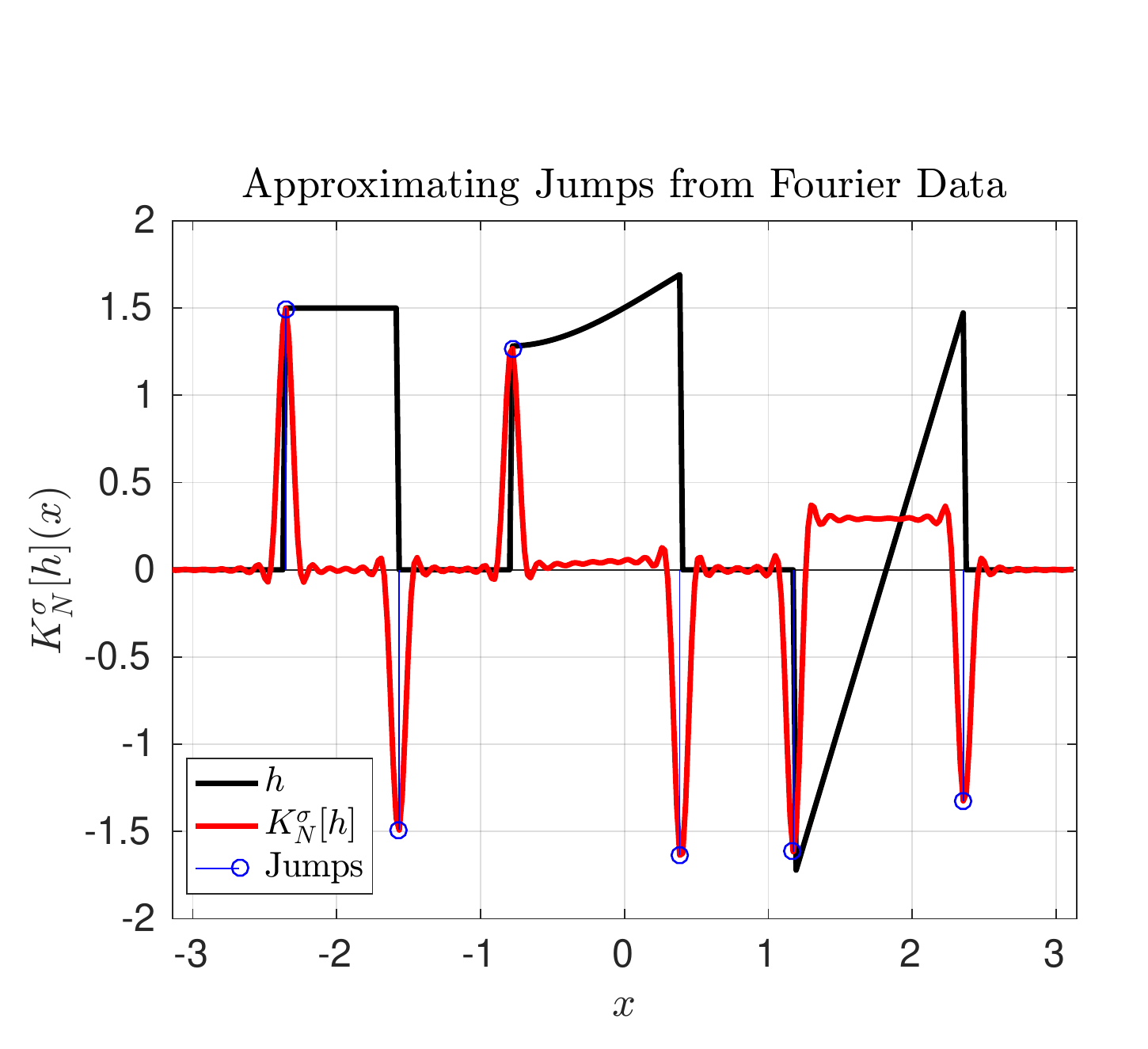} 
        \caption{Jump Approximation using $N=50$}
        \label{fig:conc_method_example}
    \end{subfigure}
    \hfill
    \begin{subfigure}[b]{0.475\textwidth}
        \includegraphics[trim = 1.2in 3.2in 1.5in 2.5in, 
            clip=true, scale = .55]{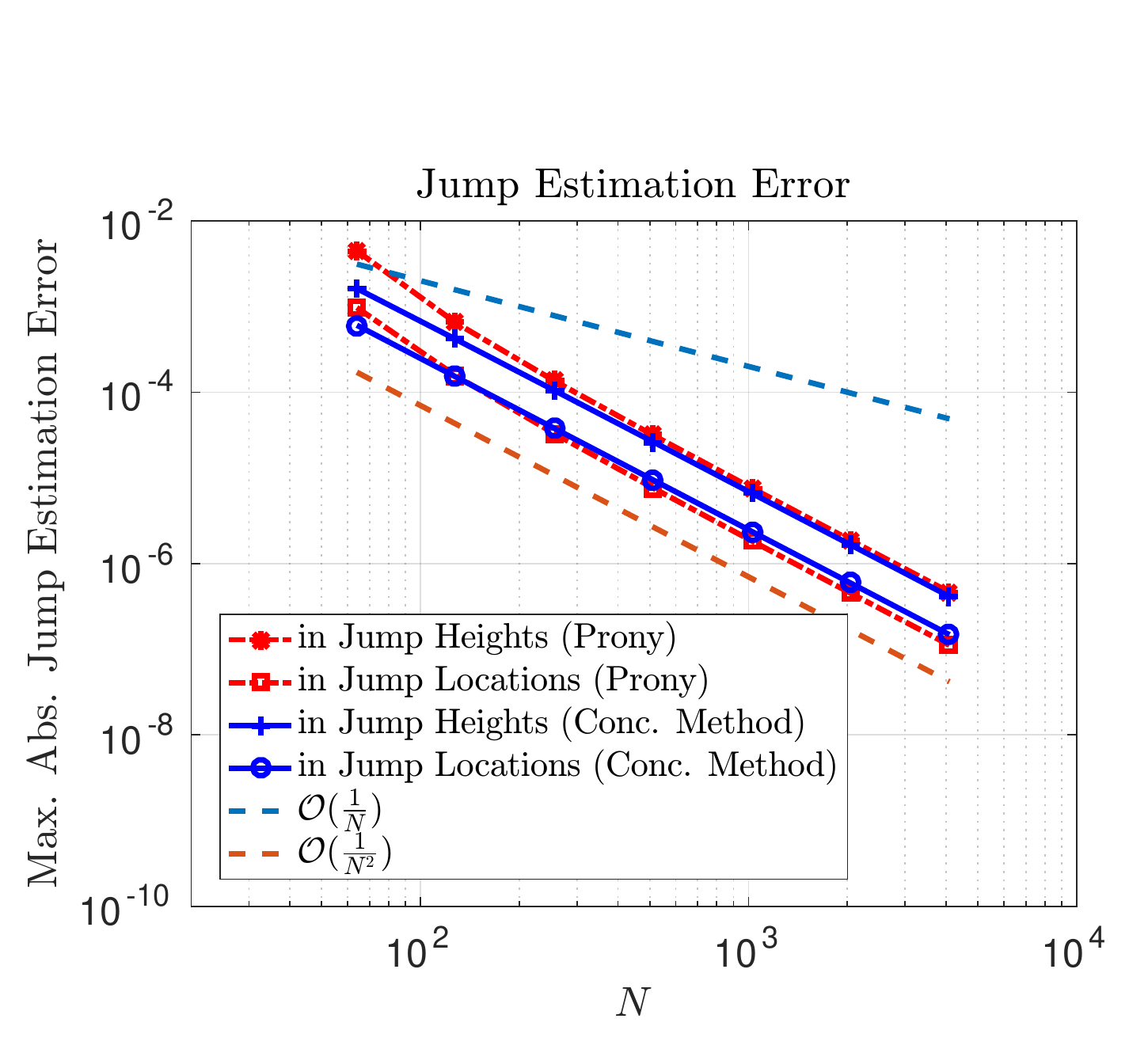} 
        \caption{Error in Approximating Jump Locations and Values}
        \label{fig:jump_conv}
    \end{subfigure}
    \caption{Jump Detection using the Concentration Method}\label{fig:conc_method}
\end{figure}

Next, we present numerical results demonstrating robustness of the concentration edge detection
method to measurement noise. Fig. \ref{fig:conc_method_noise} plots the error (averaged over 
$50$ trials) in estimating jump information from noise
corrupted Fourier coefficients, 
\[ \widehat g_k = \widehat f_k + n_k, \qquad 
            k \in [-N,N] \cap \mathbb Z, \qquad
            n_k \sim \mathcal{CN}(0, \sigma^2), \]
where $\mathcal{CN}(0, \sigma^2)$ denotes complex Gaussian additive noise of zero mean
and variance $\sigma^2$. The variance $\sigma^2$ is chosen such that 
\[ \mbox{SNR (dB)} = 10 \log_{10} \left( \frac{\|{\bf \hat f}\|_2^2}
		{(2N+1) \, \sigma^2}\right), \]
where ${\bf \hat f} = [\hat f_{-N} \dots \hat f_N]^T$ denotes the vector of true Fourier
coefficients. The figure shows that the concentration method is robust to measurement errors; 
Indeed, at high SNRs (such as in Fig. \ref{fig:noise_70}), jump locations are still recovered with
$\mathcal O(1/N^2)$ accuracy. The figure also shows that the concentration method is much more noise
tolerant than the Prony-based method of \S \ref{sec:prony}. Finally, we note that no specific
statistical constructions were used to implement Fig. \ref{fig:conc_method_noise} (such as those in 
\cite{conc_method_statistics}). We leave this to future research where we expect performance
improvements.

\begin{figure}[htbp]
    \centering
    \begin{subfigure}[b]{0.475\textwidth}
        \includegraphics[trim = 1.2in 3.2in 1.5in 2.5in, 
            clip=true, scale = .55]{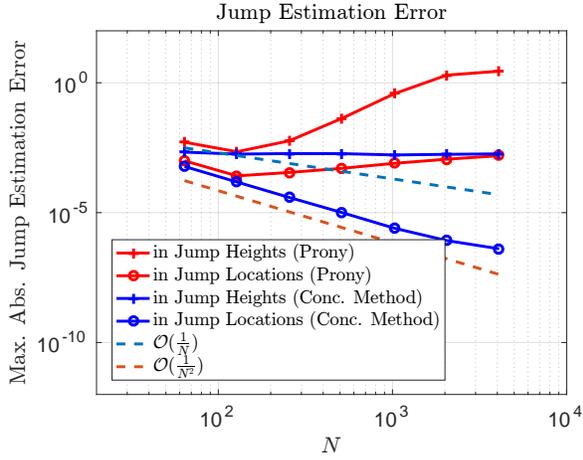} 
        \caption{Error in Estimating Jump Information at $70$dB noise}
        \label{fig:noise_70}
    \end{subfigure}
    \hfill
    \begin{subfigure}[b]{0.475\textwidth}
        \includegraphics[trim = 1.2in 3.2in 1.5in 2.5in, 
            clip=true, scale = .55]{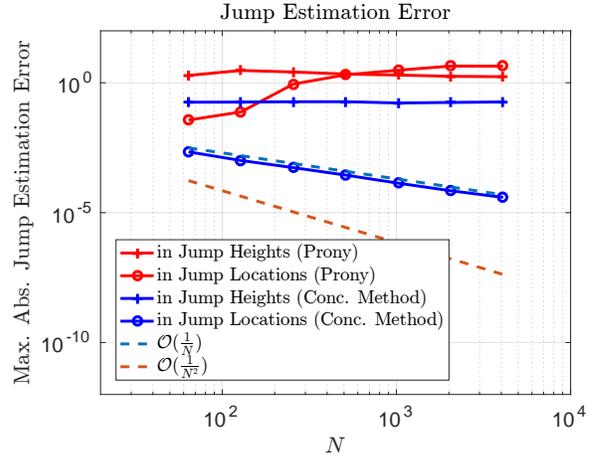} 
        \caption{Error in Estimating Jump Information at $30$dB noise}
        \label{fig:noise_30}
    \end{subfigure}
    \caption{Jump Detection using the Concentration Method in the Presence of Measurement
    Noise}\label{fig:conc_method_noise}
\end{figure}

Figure \ref{Kernel} shows the edge-augmented reconstruction using the concentration jump
estimates compared to the reconstruction using true jump information. We see that for very small
values of $N$ (Fig. \ref{fig:recon_n20}, $N=20$), the jump estimates produced by the concentration method are
inaccurate leading to errors in the edge-augmented reconstruction. However, for larger
values of $N$ (Fig. \ref{fig:recon_n40}, $N=40$), the edge augmented reconstruction using estimated
edge information is comparable to that obtained using true jump information.
(See Figure \ref{jump_error} for absolute error plot.) 

\begin{figure}[htbp]
    \centering
    \begin{subfigure}[b]{0.475\textwidth}
        \includegraphics[trim = 1.2in 3.2in 1.5in 2.5in, 
            clip=true, scale = .55]{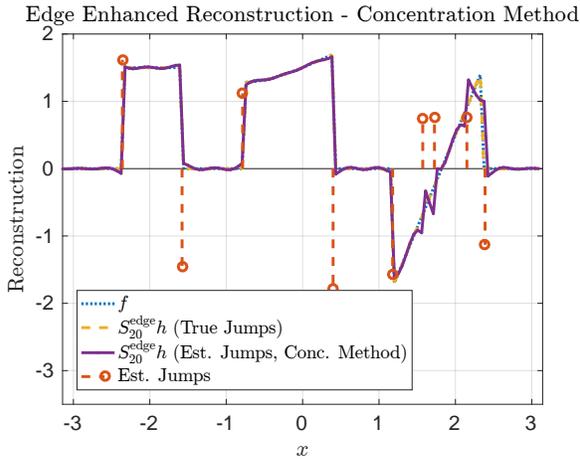} 
        \caption{Reconstruction using $|N|\leq 20$ Fourier Coefficients}
        \label{fig:recon_n20}
    \end{subfigure}
    \hfill
    \begin{subfigure}[b]{0.475\textwidth}
        \includegraphics[trim = 1.2in 3.2in 1.5in 2.5in, 
            clip=true, scale = .55]{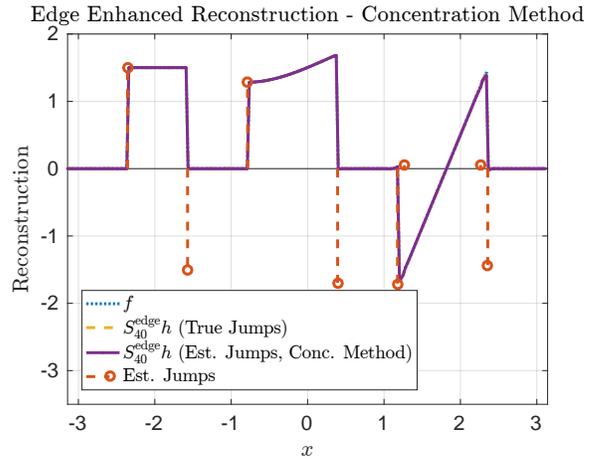} 
        \caption{Reconstruction using $|N|\leq 40$ Fourier Coefficients}
        \label{fig:recon_n40}
    \end{subfigure}
    \caption{Edge-Augmented Reconstruction of Piecewise-Smooth Function (\ref{equ:piecewise}) using the Concentration Method Jump Estimates}\label{Kernel}
\end{figure}

%

\begin{figure}[H]
    \centering
\includegraphics[scale=.8]{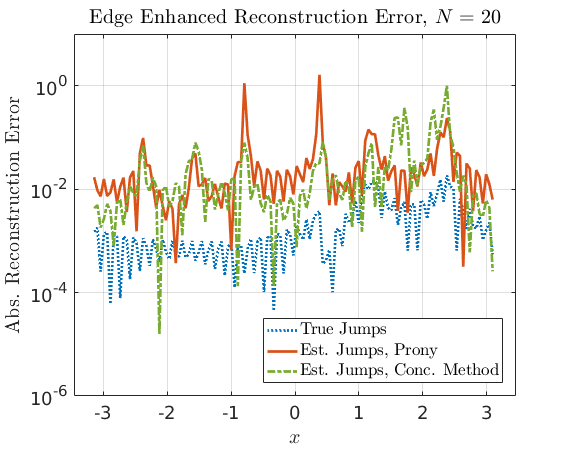}
    \caption{Absolute Error in Edge-Augmented Reconstruction of Piecewise-Smooth Function $h$
        Eqn. (\ref{equ:piecewise})}
    \label{jump_error}
\end{figure}

%
\subsection{Estimate Method Comparison}
These methods each have their own strengths and weaknesses. Prony's Method tends to produce more
accurate estimates using less coefficients, but it is highly dependent on the expected number of
jumps and is more susceptible to measurement errors. The Concentration Kernel method allows a lot more
freedom and possibility for improvement due to our ability to choose sigma, but it does not work as 
well as prony-type methods when using a very low number of Fourier coefficients.  These differences are what make it hard to choose
just a single method for obtaining estimated jump information.  Most likely a hybrid approach will work the best overall.  We suggest the use of Concentration Kernel 
techniques in the noisy regime where the number of jumps is not known well in advance (but where some additional sampling is possible), and the use of 
prony-type methods in the low sample, high SNR setting.

\section{Error Bounds for the Proposed Numerical Method with Estimated Jump Information}
\label{sec:error_bounds_est}
With any method of reconstruction from Section~\ref{Sec:NumMethJumpAugSum} there will be some error involved from the use of estimated jump information (as opposed to true jump information). Our goal is to bound this error and then compare it to the theoretical results in Section~\ref{sec:BasicThErrorBounds}. 

Let $r_j$ denote a ramp function with a jump at $x_j$ \eqref{ramp}, and let $[f](x_j)$ be an associated magnitude (both corresponding to a true jump in a function $f$ we seek to approximate).  Similarly, let $\widetilde{r_j}$ denote another ramp function with a jump at $\widetilde{x_j}$ having an associated magnitude of $a_j$.  We will assume in this section that $\widetilde{r_j}(x)$ is an approximation of $r_j(x)$ produced by one of the methods in Section~\ref{Sec:NumMethJumpAugSum} from the Fourier coefficients of $f$.  These approximate ramp functions are, furthermore, associated with an estimated jump-information based approximation of $f$ given by 
\begin{equation}
\widetilde{f}(x)=\sum_{k = -N}^{N} (\hat{f}_{k}-\widetilde{\hat{f}^{est}_{k}})e^{ikx} +\sum_{j=1}^{J} a_j \widetilde{r_j}(x),
\label{fakeEdge2}
\end{equation}
where $\widetilde{\hat{f}^{est}_{k}}$ denotes the $k^{\rm th}$ Fourier coefficient of $\sum_{j=1}^{J} a_j \widetilde{r_j}$.  In this section we seek to better understand the error $\| f - \widetilde{f} \|_2$.  Toward that end, we present the following result.

\begin{theorem} 
Let $f$ be a piecewise smooth function, $S_N^{edge}f$ be the edge-augmented Fourier sum approximation of $f$ with true jump information $\eqref{Edge 2}$, and $\widetilde{f}$ be the edge-augmented Fourier sum approximation of $f$ with estimated jump information \eqref{fakeEdge2}.  Then, if $|\widetilde{x_j}-x_j|<\epsilon$ and $|a_j-[f](x_j)|<\delta$ both hold for all $j \in \{ 1, \dots, J\}$, we have that
$$||f-\widetilde{f}||_2 \leq \sqrt{\frac{2c^2}{3N^3}} + J \delta + \sqrt{\frac{\epsilon}{2 \pi}} \left(J \delta + \sum^{J}_{j=1} \left| [f](x_j) \right|  \right)$$
where $c \in \mathbbm{R}^+$ is a constant which is independent of $N$, $\epsilon$, and $\delta$.
\label{thm:ErrEstJumpInf}
\end{theorem}



\begin{proof} 
We begin with the triangle inequality
\begin{equation*} \|f-\widetilde{f}\|_2 \leq \|f-S_N^{edge}f\|_2+\|S_N^{edge}f-\widetilde{f}\|_2.
\end{equation*} 
Theorem~\ref{fsnedgeerror} can now be applied to bound the first term above.  Doing so we learn that
\begin{equation}\label{tri} \|f-\widetilde{f}\|_2\leq\sqrt{\frac{2c^2}{3N^3}} + \|S_N^{edge}f-\widetilde{f}\|_2.
\end{equation} 
The remainder of the proof is now dedicated to bounding the second term in \eqref{tri}.  

Using \eqref{Edge 1} together with a similar form for $\tilde{f}$ we can see that 
\begin{align*}
\|S_N^{edge}f-\widetilde{f}\|^2_2 &= \left\| \sum_{|k| > N} \left( \hat f_{k}^{est} - \widetilde{\hat{f}^{est}_{k}} \right)e^{ikx} \right\|^2_2 & \\
& = \sum_{|k| > N} \left| \hat f_{k}^{est} - \widetilde{\hat{f}^{est}_{k}} \right|^2 \hspace{.7in} &\textrm{(by Parseval's Theorem)}\\
& \leq \sum^{\infty}_{k = -\infty} \left| \hat f_{k}^{est} - \widetilde{\hat{f}^{est}_{k}} \right|^2 &\\
& = \left\| \sum_{j=1}^{J}[f](x_{j})r_{j} - \sum_{j=1}^{J} a_j \widetilde{r_j} \right\|^2_2  &\textrm{(by Parseval's Theorem)} \\
& = \left\| \sum_{j=1}^{J} \left([f](x_{j})r_{j} - a_j \widetilde{r_j} \right) \right\|^2_2. &
\end{align*}
A second use of the triangle inequality now reveals that
\begin{equation}
\|S_N^{edge}f-\widetilde{f}\|_2 \leq \sum_{j=1}^{J} \left\| [f](x_{j})r_{j} - a_j \widetilde{r_j} \right\|_2.
\label{equ:LastThmBound}
\end{equation}

In order to upper bound the righthand side of \eqref{equ:LastThmBound} we must now consider $$d_j := [f](x_{j})r_{j} - a_j \widetilde{r_j}$$ for each $j \in \{ 1, \dots, J\}$.  Appealing again to the ramp function definition in \eqref{ramp}, one can see that
\begin{align*}
d_j(x) ~=~& \chi_{(-\pi,\min(x_j,\widetilde{x_j}))}(x) \left[ \frac{a_j - [f](x_{j})}{2} \right] +  \widetilde{\chi}_{(x_j,\widetilde{x_j})}(x) \left[ \left( \frac{[f](x_{j})+a_j}{2} \right) {\rm sign}(\widetilde{x_j} - x_j) \right]\\ 
&+ \chi_{(\max(x_j,\widetilde{x_j}),\pi]}(x) \left[ \frac{[f](x_{j})-a_j}{2} \right] + \left( \frac{a_j - [f](x_{j})}{2 \pi} \right) x
\end{align*}
where 
$$\widetilde{\chi}_{(x_j,\widetilde{x_j})}(x) :=\begin{cases} 1, & {\rm if}~\widetilde{x_j}>x_j \;\;{\rm and}\;\; x\in[x_j,\widetilde{x_j}] \\
1, & {\rm if}~x_j > \widetilde{x_j} \;\;{\rm and}\;\; x\in[\widetilde{x_j},x_j] \\
0, & {\rm else}
\end{cases}.$$
Thus, 
 $$|d_j(x)|^2 \leq \begin{cases} \delta^2, & {\rm if}~x < \min(\widetilde{x_j},x_j) \;\;{\rm or}\;\; x > \max(\widetilde{x_j},x_j) \\
\left( \delta + \left| [f](x_{j}) \right| \right)^2, &  {\rm else}
\end{cases}.$$

With these inequalities in hand we can now continue to bound \eqref{equ:LastThmBound} by
\begin{align*}
\|S_N^{edge}f-\widetilde{f}\|_2 ~&\leq~ \sum_{j=1}^{J} \sqrt{ \frac{1}{2 \pi} \int^{\pi}_{-\pi} |d_j(x)|^2 ~dx }\\
& \leq~\sum_{j=1}^{J} \sqrt{ \delta^2 + \frac{\epsilon}{2 \pi} \left( \delta + \left| [f](x_{j}) \right| \right)^2} \\
& \leq~\sum_{j=1}^{J} \left[ \delta + \sqrt{\frac{\epsilon}{2 \pi}} \left( \delta + \left| [f](x_{j}) \right| \right) \right] \\
&=~J \delta + \sqrt{\frac{\epsilon}{2 \pi}} \left( J \delta + \sum_{j=1}^{J} \left| [f](x_{j}) \right| \right).
\end{align*}
Combining this last inequality with \eqref{tri} finishes the proof.
\end{proof}

Comparing Theorem~\ref{thm:ErrEstJumpInf} to Corollary~\ref{cor:PS2normError} we can see that $\widetilde{f}$ from \eqref{fakeEdge2} will approximate piecewise smooth functions with jumps better than $S_Nf$ will as long as $\delta$ is $o(1/\sqrt{N})$ and $\epsilon$ is $o(1/N)$.  In fact, the numerical behavior is much better than this (suggesting, among other things, that Theorem~\ref{thm:ErrEstJumpInf} can probably be improved).

\begin{figure}[hbtp]
\centering
\includegraphics[scale=.8]{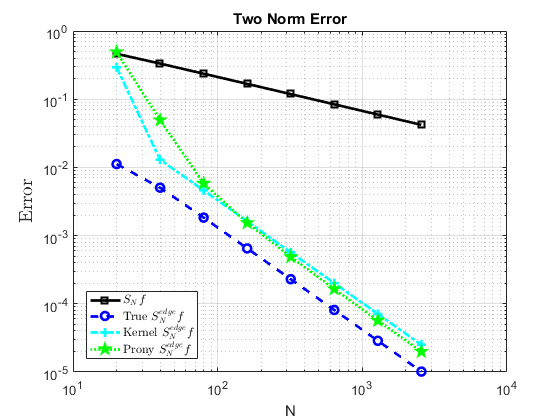}
\caption{Two norm error comparison for all means of reconstruction}
\label{norm_graph}
\end{figure}

Figure \ref{norm_graph} shows the error from Figure \ref{true and standard two norm error} but now also includes the error from using estimated jump information via Prony's Method as well as the trigonometric concentration kernel. Once again, the slope of each line in Figure \ref{norm_graph} is important to note. Since we won't have true jump information in applications, we want to see if our estimated information can still give us a similar degree of accuracy as the error expected using true jump information. At lower values of $N$, the error is closer to that of the standard reconstruction, but as $N$ increases the slopes of the reconstructions from Prony's Method and the concentration kernel start to match that of the slope from true jump information. This means the decay for the estimated jump information is a great improvement over the standard reconstruction (and, indeed, near optimal for larger $N$).

\section{Preliminary Two-Dimensional Results}
\label{sec:2D}
We conclude by presenting preliminary results demonstrating the extension of the proposed method to
the two-dimensional case. 
Consider the reconstruction of a two-dimensional $2\pi$-periodic function $f$ given its 
Fourier series coefficients 
\begin{equation}
    \hat f_{k,\ell}=\frac{1}{4\pi^2}\int_{-\pi}^{\pi} \int_{-\pi}^{\pi} f(x,y) 
e^{-ikx}e^{-i \ell y}dxdy, \qquad (k, \ell) \in [-N, N]^2\cap \mathbb Z^2.
\end{equation}
If $f$ is piecewise-smooth, the 2D Fourier partial sum reconstruction 
\begin{equation}
    S_{N,M} f(x,y)=\sum_{|k| \leq N}\sum_{|\ell | \leq M} \hat f_{k, \ell } e^{ikx}e^{i \ell y}
\end{equation}
\begin{figure}[hbtp]
    \centering
    \includegraphics[scale = 0.5]{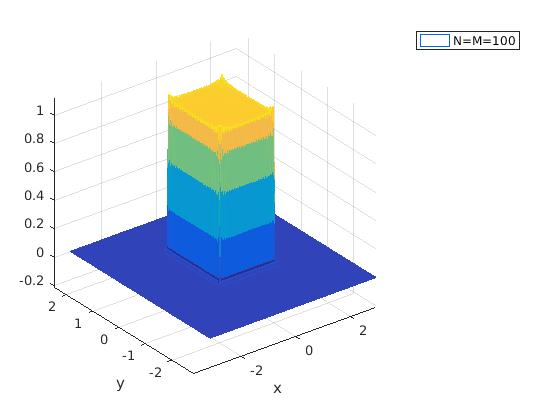}
    \caption{2D Fourier Partial Sum Reconstruction of the Box Function using $100 \times 100$
    coefficients.}
    \label{fig:2dbox1}
\end{figure}
suffers from Gibbs artifacts as in the one dimensional case. For example, Fig. \ref{fig:2dbox1}
plots the Fourier partial sum reconstruction of the ``box'' function 
\begin{equation}
    f_1(x,y) = \left \{ 
        \begin{array}[h]{lc}
            1, & \text{if } -1\leq x,y\leq 1, \\
            0, & \text{else}
        \end{array}
        \right.
    \label{eq:2dbox}
\end{equation}
using $|N|, |M| \leq 100$ Fourier series coefficients. Despite the large number of coefficients used
in the reconstruction, note the significant Gibbs artifacts along the edges of the function.

To mitigate these artifacts, we consider the following two-dimensional extension of the edge-augmented
partial sum procedure of \S \ref{sec:JumpAugPartSumRec}. Herein, we use the notation 
$g_{r_j}(x):= f(x, y_j)$ and $g_{c_j}(y):=f(x_j, y)$ to denote the row and column cross-sections of
$f$ respectively. Without loss of generality, and for ease of notation, we assume that $N=M$.
Furthermore, we assume that reconstruction is performed on a $M\times M$ equispaced grid in image
space.
\begin{codebox}
\Procname{$\proc{Two-Dimensional Edge-Augmented Fourier Sum}$}
\li Given: Fourier coefficients $\widehat f_{k,l}$ for $|k|,|l|\leq N$
\li Define oversampled grid in the $y$-dimension, $y_j = -\pi+j\frac{2\pi}{M_{over}}, \; j =
0,\dots,M_{over}-1$ where $M_{over}\geq M$ \\
\li Reconstruct along the rows of $f$ \Do
\li \For $j \in \{0,1,\dots,M_{over}-1\}$ \Do
\li         Approximate Fourier coefficients of the (row) cross-section \\ 
                \qquad \qquad \qquad \qquad \qquad \qquad  
                $\displaystyle (\widehat g_{r_j})_k \approx \sum_{|l|\leq N} 
                \widehat f_{k,l} e^{ily_j}, \;|k| \leq N$. 
\li         Estimate jump locations and associated jump values 
            $\{(\widetilde x_p^j, a^j_p)\}_{p=1}^{L_x^j}$ 
            along the cross-section \\ \qquad \qquad
            using \S \ref{sec:prony} or \S \ref{sec:conc}. (Here, $L_x^j$ denotes the number of 
                jumps in the cross-section.)
\li         Compute the edge-augmented Fourier partial sum approximation $\widetilde f(x, y_j)$ of 
            \\ \qquad \qquad the (row) cross-section using (\ref{fakeEdge2}).
        \End  \End \\
\li Refine reconstruction along the columns of $f$ \Do
\li \For $j \in \{0,1,\dots,M-1\}$ \Do
\li         Approximate Fourier coefficients of the (column) cross-section using quadrature \\
                \qquad \qquad \qquad \qquad \qquad \qquad  
                $\displaystyle (\widehat g_{c_j})_\ell \approx \frac 1 {M_{over}} \sum_{k=0}^{M_{over}-1} 
                \widetilde f(x_j, y_k) e^{-i\ell y_k}, \;|\ell| \leq N$.
\li         Estimate jump locations and associated jump values 
            $\{(\widetilde y_p^j, b^j_p)\}_{p=1}^{L_y^j}$ 
            along the cross-section \\ \qquad \qquad
            using \S \ref{sec:prony} or \S \ref{sec:conc}. (Here, $L_y^j$ denotes the number of 
                jumps in the cross-section.)
\li         Compute edge-augmented Fourier partial sum approximation $\widetilde f(x_j, y)$ of 
            \\ \qquad \qquad the (column) cross-section using (\ref{fakeEdge2}).
\li         Output the grid values from (column) cross-section estimates $\widetilde f(x_j, y)$ as final grid estimates.
        \End
\end{codebox}

Some representative results can be seen in Fig. \ref{fig:2d} which plots the reconstruction of the function  
\begin{align*}
    f_2(x,y) & = 0.75 \ {\bf 1}_{  \left[-9/4,-1/4\right] \times
                                    \left[-5/2, -1/2\right]} \quad + \\
             & \qquad 0.50 \ {\bf 1}_{\{(x,y) \in \mathbb R^2 \ | \  
                                    (x-1/2)^2+(y-1)^2\leq1\}} \quad + \\
             & \qquad 0.35 \ {\bf 1}_{\{(x,y) \in \mathbb R^2 \ | \  
                                    (x-5/4)^2+(y+5/4)^2\leq 1/4\}}
\end{align*}
(here ${\bf 1}_\mathcal A$ denotes the indicator function of $\mathcal A \subset \mathbb R^2$) using
its two-dimensional (continuous) Fourier series coefficients $(\widehat{f_2})_{k,l}$ for $|k|, |l| \leq
25$. Note that the Fourier coefficients of this function can be evaluated in closed form since
the Fourier coefficients of the box and circle are known in closed form.
We can see that the standard Fourier partial sum reconstruction in Fig. \ref{fig:2d}(a) shows
significant Gibbs oscillatory artifacts, while the proposed method in Fig. \ref{fig:2d}(b) is much
more accurate. This is especially evident in the error plots of the respective
reconstructions -- note the large errors along the edges of the features of $f_2$ in the Fourier
partial sum, while the edge-augmented reconstruction shows significantly reduced errors and
artifacts. Errors are reported in terms of the peak signal to noise ratio (PSNR); let ${\mathbf f}$ and 
$\widetilde {\mathbf f}$ denote a two-dimensional $N\times N$ grid function and its approximation respectively. Then, 
\[ PSNR = 20 \log_{10} \left( \frac{N\; \| \text{vec}({\mathbf f})\|_\infty}{\|{\mathbf f}-\widetilde {\mathbf f}\|_F }\right). \]
As with the one-dimensional case, the concentration kernel edge detection method
is used to estimate edge locations and heights from the given Fourier data. 
\begin{figure}[htbp]
    \centering
    \begin{subfigure}[b]{0.475\textwidth}
        \includegraphics[clip=true, trim=1cm 8.2cm 1cm 8.95cm, width=9cm]{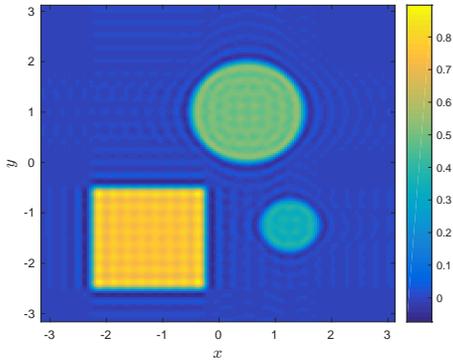} 
        \caption{Standard Fourier Reconstruction}
        \label{fig:2d_four}
    \end{subfigure}
    \hfill
    \begin{subfigure}[b]{0.475\textwidth}
        \includegraphics[clip=true, trim=1cm 8.2cm 1cm 8.95cm, width=9cm]{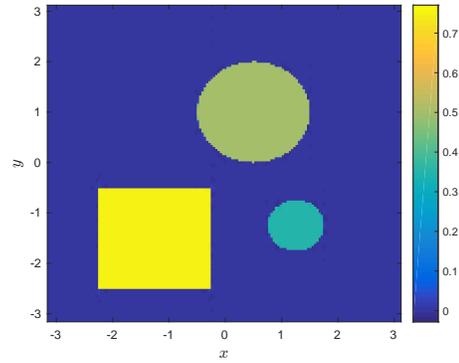} 
        \caption{Proposed Method}
        \label{fig:2d_proposed}
    \end{subfigure} \vspace{0.1in} \\
    \begin{subfigure}[b]{0.475\textwidth}
        \includegraphics[clip=true, trim=1cm 8.2cm 1cm 8.95cm, width=9cm]{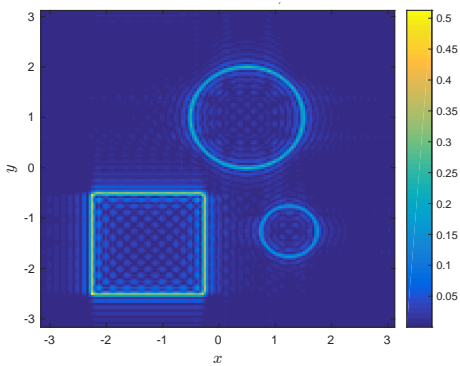} 
        \caption{Error using Fourier Reconstruction, PSNR = $26.97$ dB}
        \label{fig:2d_four_error}
    \end{subfigure}
    \hfill
    \begin{subfigure}[b]{0.475\textwidth}
        \includegraphics[clip=true, trim=1cm 8.2cm 1cm 8.95cm, width=9cm]{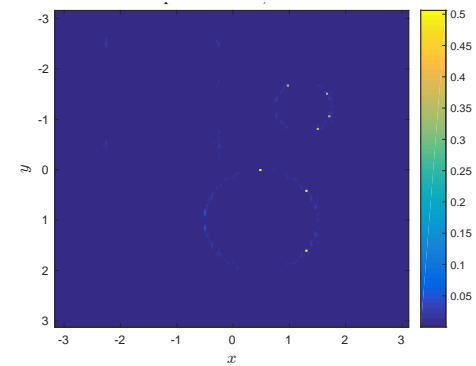} 
        \caption{Error using Proposed Method, PSNR = $42.76$ dB.}
        \label{fig:2d_proposed_error}
    \end{subfigure}
    \caption{2D Edge-Augmented Fourier Reconstruction using Fourier coefficients $(\widehat{f_2})_{k,l}$ for $|k|, |l| \leq 25$ }
    \label{fig:2d}
\end{figure}

\section{Conclusion}
\label{sec:conclusion}
\indent In this paper we discuss reconstruction through partial sum approximations and further, seek
to reduce the Gibbs phenomenon. We use examples with incorporated true jump information and
estimated jump information. We additionally explore two methods (namely, Prony's Method and the
Concentration Kernel Method) to accurately estimate jump heights and locations. We show how these
methods can be applied in the one dimensional case while also providing preliminary two dimensional 
empirical results. Additionally, we derive theoretical error bounds for these methods of reconstruction 
for the one dimensional case.

There are several avenues for future investigation.  
First, the noise produced by the MRI machine itself needs to be accounted for in determining the
practical robustness of our methods.  Second, more advanced Concentration Kernel methods could be
developed according to \cite{ImprovedConcentrationKernels} and evaluated against the current
results.
Similarly, the analytical error bounds and better methods for the two-dimensional case could be developed. See references in, e.g., \cite{AlgebraicReconstruction} for more on existing techniques.





\section{Acknowledgements} 
\indent We would like to thank 
Dr. Tsvetanka Sendova for her insight and guidance throughout this project.  We also extend our gratitude to 
Michigan State University, Lyman Briggs College and, more specifically, Dr. Robert Bell for organizing and hosting the SURIEM (Summer Undergraduate Research Institute in Experimental Mathematics) program this summer.  We would also like to thank our home colleges/universities: Wellesley College, Albion College, Goucher College, and the University of Virginia. Lastly, a great thanks to the National Security Agency and National Science Foundation for funding the Mathematics REU. 

Project sponsored by the National Science Foundation under Grant Number DMS-1559776.

Project sponsored by the National Security Agency under Grant Number H98230-16-1-0031.

\bibliographystyle{plain}
\bibliography{refs}

\end{document}